\documentclass{birkmult}

\usepackage{amsfonts}
\usepackage{amsmath,amscd}
\usepackage{xcolor}

 \newtheorem{theorem}{Theorem}[section]
 \newtheorem{corollary}[theorem]{Corollary}
 \newtheorem{lemma}[theorem]{Lemma}
 \newtheorem{proposition}[theorem]{Proposition}
 \theoremstyle{definition}
\newtheorem{definition}[theorem]{Definition}
\newtheorem{conjecture}[]{Conjecture}
 \theoremstyle{remark}
 \newtheorem{remark}[theorem]{Remark}
 \newtheorem{example}{Example}

\def\scal#1#2{\langle #1, #2\rangle}
\def\R#1{\mathbb{R}^{#1}}

\def\TOP{\mathrm{t}}

\DeclareMathOperator{\DIV}{\mathrm{div}}
\DeclareMathOperator{\trace}{\mathrm{tr}}

\begin{document}

\title{Minimal cubic cones via Clifford algebras}



\author{Vladimir Tkachev}

\address{Mathematical department, Royal Institute of Technology, S-10044, Stockholm, Sweden}

\email{tkatchev@kth.se}

\begin{abstract}
In this paper, we construct two  infinite families of algebraic minimal cones in $\R{n}$. The first family consists of  minimal cubics given explicitly in terms of the Clifford systems. We show that the classes of congruent minimal cubics are in one to one correspondence with those of geometrically equivalent Clifford systems. As a byproduct, we prove that for any $n\ge4$, $n\ne 16k+1$, there is at least one minimal cone in $\R{n}$ given by an irreducible homogeneous cubic polynomial. The second family consists of minimal cones in $\R{m^2}$, $m\ge2$, defined by an irreducible homogeneous polynomial of degree $m$. These examples provide particular answers to the questions on algebraic minimal cones in $\R{n}$ posed by Wu-Yi Hsiang in the 1960's.
\end{abstract}


\maketitle

\section{Introduction}

In 1916, S.~Bernstein proved his famous theorem asserting that any entire solution $u=u(x_1,x_2)$ of the minimal surface equation
\begin{equation}\label{min}
\mathrm{div} \frac{\nabla u}{\sqrt{1+|\nabla u|^2}}=0,
\end{equation}
must be an affine function. In general, the left hand side of (\ref{min}) is the mean curvature of the graph $x_{n}=u(x_1,\ldots,x_{n-1})$ in $\R{n}$. The question whether the Bernstein result holds true in any dimension $n\ge 3$ (Bernstein's problem) was  a long standing problem until J.~Simons \cite{Simons} proved that the Bernstein property holds in lower dimensions $n\le 8$ and E.~Bombieri, E.~Di Giorgio and E.~Giusti in \cite{BGG} constructed a non-affine minimal graph over $\R{9}$. An important role in establishing of Bernstein's problem and the constructing of non-affine examples for $n\ge 9$ played Simon's cone $\{(x, y)\in \R{4}\times \R{4}: |x|^2=|y|^2\}$ in $\R{8}$ and its generalizations in $\R{p+q}$ given by the implicit equation
\begin{equation}\label{ClSim}
(q-1)(x_1^2+\ldots+x_p^2)+(p-1)(y_{1}^2+\ldots+y_{q}^2)=0, \quad p,q\ge 2.
\end{equation}
We refer the interested reader to \cite{SS}, \cite{Simon89} for the further details, and mention a very recent discussion of the breakdown of Bernstein's theorem and geometry of quadratic minimal cones in context of critical dimensions for the stable  cone solution of brane in the brane–black-hole system \cite{Gibbons}, \cite{Frolov}.

The Clifford-Simons cones being defined by a quadratic equation are the simplest  examples of \textit{algebraic} minimal cones. It is well known that any minimal quadratic cone can be brought into the form (\ref{ClSim}) in some orthogonal coordinates in $\R{n}$. In particular, any quadratic minimal cone is completely determined by a non-ordered integer pair $(p,q)$.

On the other hand, finding classification of minimal algebraic cones of higher degrees remains a long-standing problem \cite{SimonL}, \cite{Hsiang67}, \cite{Fomenko}.  Even in the case of minimal \textit{cubic} cones there are only few  examples known. For our further convenience, we give a short description of these examples.

\begin{example}\label{ex1}
By using the representation theory of the compact Lie groups $SO(3)$ and $SU(3)$, one gets two homogeneous minimal cubic cones in $\R{4}$ and $\R{7}$ respectively. For example, the defining equation of the first cone is $x_3(x_1^2-x_2^2)+2x_4x_1x_2=0$. This cone is also a member of the Lawson family \cite{Lawson} of compact minimal surfaces in the unit sphere $\mathbb{S}^3\subset \R{4}$.
\end{example}

\begin{example}\label{ex2}
Another series of minimal cubic cones is obtained from four isoparametric surfaces with three constant principal curvatures discovered  by \'{E}.~Cartan \cite{Cartan} in 1939. These  correspond to  standard Veronese embeddings of a projective plane $\mathbb{F}_d P^2$ into the unit sphere in $\R{3d+2}$, $d=1,2,4,8$, where $\mathbb{F}_d$ is one of the four only possible classical division algebras: $\mathbb{F}_1=\R{}$, $\mathbb{F}_2=\mathbb{C}$, $\mathbb{F}_4=\mathbb{H}$ and $\mathbb{F}_8=\mathbb{O}$. The corresponding defining polynomials are given explicitly by (cf. \cite[p.~34]{Cartan})
\begin{equation*}\label{CartanFormula}
\begin{split}
f_d(x)& =x_{n}^3-3x_{n}x_{n-1}^2+\frac{3}{2}x_n(X_0\bar X_0+X_1\bar X_1-2X_2\bar X_2)\\
&+\frac{3\sqrt{3}}{2}x_{n-1}(X_0\bar X_0-X_1\bar X_1)
+\frac{3\sqrt{3}}{2}((X_0X_1)X_2+\bar X_2(\bar X_1\bar X_0)),
\end{split}
\end{equation*}
where $x=(X_0,X_1,X_2,x_{n-1},x_{n})$, the vectors $X_k=(x_{kd+1},\ldots,x_{kd+d})$ are identified with the corresponding elements of $\mathbb{F}_d$, $k=0,1,2$, and $\bar X$ denotes the conjugate  of $X$ in $\mathbb{F}_d$. It is well known that $f^{-1}_d(0)$ are minimal cubic cones in $\R{n}$, $n=3d+2$ (thus, $n=5,8,14$ and $26$).
\end{example}

\begin{example}\label{ex3}
W.~Y.~Hsiang \cite{Hsiang67} gave an elegant construction of two minimal cubics, in $\R{9}$ and $\R{15}$ respectively. The first example  exploits some special properties of the orthogonal invariants of the space $\mathfrak{G}'(4,\R{})$ of quadratic forms of 4 real variables with trace zero. By identifying $\mathfrak{G}'(4,\R{})$ with $\R{9}$, the defining polynomial of  the first cubic is given by $b_3(Y)=0$, where
$$
\det (Y-tI_{\R{4}})=t^4+b_2(Y)t^2+b_3(Y)t+\det Y,\quad Y\in M,
$$
and $I_V$ stands for the identity operator on $V$. We discuss an explicit representation of this cubic in Section~\ref{sec:det} below. The second cubic is similarly constructed by making use of the space $\mathfrak{G}'(4,\mathbb{C}{})$ of Hermitian forms of 4 complex variables with trace zero.
\end{example}

In this paper, we construct two new families of minimal cones. We show that any Clifford system $A_0,A_1,\ldots,A_q$ in $\R{2m}$ generates an irreducible minimal cubic cone in $\R{2m+q+1}$; see explicit formulas in Section~\ref{sec:3} below. This yields a partial answer to the Problem~1 posed earlier by Hsiang in \cite{Hsiang67}. Namely, we show that for any $n\ge 4$, $n\ne 16k+1$, there is at least one irreducible minimal cubic in $\R{n}$. In Section~\ref{sec:4} we show that the congruence classes of the constructed minimal cubics are in one-to-one correspondence with those of geometrically equivalent Clifford systems.

In Section~\ref{sec:det} we describe another family of minimal cones given by irreducible homogeneous polynomials of arbitrary high degree. More precisely, for any integer $n\ge 2$ we construct a  minimal cone in $\R{n^2}$ given by an irreducible homogeneous polynomial of degree  $n$. If $n=3$, the corresponding cone agrees under an isometry with the Hsiang minimal cubic cone in $\R{9}$ mentioned in Example~\ref{ex3} above.

\section{Preliminaries}\label{sec:pre}

Let $f\in \R{}[x_1,\ldots, x_n]$ be a homogeneous irreducible polynomial with real coefficients of degree $s\ge 1$. If we assume that $\mathcal{F}(f)$ contains regular points of $f$, then $\mathcal{F}(f):= f^{-1}(0)$ defines an $(n-1)$-dimensional cone in $\R{n}$. Then it is well known (see, for instance, \cite{Hsiang67}) that the cone $\mathcal{F}(f):= f^{-1}(0)$ is minimal if and only if
\begin{equation}\label{Lgamma1}
L(f)\equiv 0 \mod f,
\end{equation}
where the normalized mean curvature operator $L$ is defined by
\begin{equation}\label{Lgamma}
L(f):=|\nabla f|^{3} \DIV \frac{\nabla f}{|\nabla f|}= |\nabla f|^2\Delta f-
\sum_{i,j=1}^n f_{x_i} f_{x_j} f_{x_ix_j}.
\end{equation}
For linear forms (i.e. $s=1$) one has  $L(f)\equiv 0$, which reflects the well-known fact that hyperplanes in $\R{n}$ have zero mean curvature.

Note that the above congruence also is well-defined for general homogeneous polynomials $f$, thus we have the following definition.

\begin{definition}\label{def1}
A homogeneous polynomial $f$ satisfying  (\ref{Lgamma1}) is called an \textit{eigenfunction} of $L$. If $\deg f=3$ it is called a \textit{minimal cubic}. The ratio $\lambda(f):=L(f)/f$ will be called the \textit{weight}  of $f$.
\end{definition}

For a minimal cubic, its weight $\lambda(f)$ is a quadratic form, hence its diagonal form has an invariant meaning (recall that the operator $L$ is invariant under orthogonal substitutions  \cite[Lemma~2]{Hsiang67}). In this paper, we are primarily  interested in the case when
\begin{equation}\label{eigg}
\lambda(f)=c|x|^2, \quad c\in \R{}.
\end{equation}
Any minimal cubic $f$ satisfying (\ref{eigg}) will be called a \textit{radial} minimal cubic. Notice that for a radial minimal cubics, the weight function $\lambda(f)$ is rotationally invariant and it contains a complete information about all variables of $f$, which makes it possible to define the dimension of $f$ as the dimension of the vector $x$.

\begin{remark}
The notion of dimension for general eigenfunctions of $L$ is not well defined  because any eigenfunction in $\R{n}$ also is (by adding `redundant variables') an eigenfunction in a larger $\R{n+k}$.
\end{remark}

\begin{definition}
Two homogeneous polynomials $f_1$ and $f_2$ are called \textit{congruent} if $f_1(x)=cf_2(Ux)$, where $U$ is an orthogonal endomorphism of $\R{n}$ and $c\in \R{}$, $c\ne0$.
\end{definition}

\begin{proposition}\label{pr:1}
Two irreducible cubics $f_1$ and $f_2$ in $\R{n}$ are congruent if and only if the corresponding cones $\mathcal{F}(f_1)$  and $\mathcal{F}(f_2)$ are congruent.
\end{proposition}

The proof is based on the following lemma.

\begin{lemma}\label{lem:milnor}
Let $f$ be an irreducible cubic,  $f\not\equiv 0$. Then $\mathcal{F}(f)$ contains a regular point of $f$.
\end{lemma}

\begin{proof}
First we notice that in some orthogonal coordinates, $f$ can be written as follows:
\begin{equation}\label{gap}
f(x)=ax_n^3+x_n\phi(x_1,\ldots,x_{n-1})+\psi(x_1,\ldots,x_{n-1}), \quad a\ne0.
\end{equation}
Indeed, the restriction of $f(x)$ on the unit sphere $|x|^2=1$ is continuous, hence it attains its maximum at some point $x^0$ (since $f(x)$ is a non-identically zero odd function, the maximum is strictly positive). By Lagrange principle, the gradient of $f$ is collinear to the unit vector $\nabla |x|^2$ at $x^0$, i.e. $\nabla f(x^0)=a x^0$. We have by the homogeneity of $f$,
$$
a=\scal{\nabla f(x^0)}{x^0}=3f(x^0)> 0.
$$
Setting $e_n=x^0$ and completing $e_n$ to an orthonormal basis of $\R{n}$, one can easily see that $f(x)$ takes the required gap form (\ref{gap}) in the new coordinates.

We shall proceed by contradiction. Suppose that $\mathcal{F}(f)$ contains no regular points of $f$, i.e. $f=0$ implies $\nabla f=0$. Since $f$ is irreducible, $\psi\not\equiv 0$, hence there exists   $u=(u_1,\ldots,u_{n-1})\ne 0$ such that $\psi(u)\ne0$. Then the cubic polynomial $P(t)\equiv at^3+bt+c$ has a real root $t_1\ne0$, where $b=\phi(u)$ and $c=\psi(u)$. Denote by $T$ the set of all real roots of $P(t)$. Then for any $t_i\in T$, the point $(u_1,\ldots,u_{n-1},t_i)\in \mathcal{F}(f)$, thus $\nabla f(u_1,\ldots,u_{n-1},t_i)=0$. The latter implies that the partial derivative $\partial_{x_n}f(u_1,\ldots,u_{n-1},t_i)=0$. Thus $P'(t_i)=0$ for any $t_i\in T$, that is any real zero of the polynomial $P(t)$ must be also a zero of its derivative $P'(t)$. But it is possible only if $P(t)\equiv at^3$, which  contradicts to $c=\psi(u)\ne 0$. The lemma is proved.
\end{proof}

\begin{proof}[Proof of Proposition~\ref{pr:1}]
It suffices to prove the `only if' part.
Suppose that the cones $\mathcal{F}(f_1)$ and $\mathcal{F}(f_1)$ are congruent, i.e. $\mathcal{F}(f_2)=U\mathcal{F}(f_1)$ for some orthogonal endomorphism $U$ of $\R{n}$. Consider the new cubic $g(x)=f_2(Ux)$. Then $g(x)= 0$ whenever $f_1(x)=0$. By Lemma~\ref{lem:milnor},  the cone $\mathcal{F}(f_1)$ contains a regular point. Applying the real Nullstellensatz for the algebraic set $\mathcal{F}(f_1)$ given by a single polynomial equation (see, for instance, Lemma~2.5 in \cite{Milnor}), we conclude that   $g(x)$ must be a multiple of $f_1$. Since $\deg g=\deg f_1=3$, we conclude that $g=cf_1(x)$ for some real $c\ne 0$. This shows that $cf_1(x)=f_2(Ux)$, the theorem follows.
\end{proof}

\section{Clifford minimal cubics}\label{sec:3}


We shall exploit the well-known fact that any Clifford algebra has a linear representation in a Euclidean space $\R{2m}$ such that all algebra generators act as orthogonal endomorphisms. This link between the representation theory of Clifford algebras and the so-called Clifford systems can be described as follows \cite{Baird}. Recall that a $q$-tuple $\mathcal{A}=(A_0,\ldots, A_q)$, $q\ge1$, of symmetric endomorphisms of $\R{2m}$ is called a \textit{(symmetric) Clifford system} on $\R{2m}$, or $ \mathcal{A}\in \mathrm{Cliff}(\R{2m},q)$, if
\begin{equation}\label{Cliffsys}
A_iA_j+A_jA_i=2\delta_{ij}\cdot I_{\R{2m}}, \quad 0\le i,j\le q.
\end{equation}
Recall that by $I_V$ we denote the identity operator on $V$.

\begin{remark}\label{r1}
Notice that all $A_i$ are orthogonal matrices and they are necessarily trace free. Indeed, since the trace is invariant under cyclic permutations, we have from (\ref{Cliffsys}) for $i\ne j$: $\trace A_i=-\trace A_jA_iA_j=-\trace A_iA_j^2=-\trace A_i$.
\end{remark}

Given a finite collection of matrices $A_i\in \R{s\times s}$, $0\le i\le q$, and a vector $z=(z_0,z_1,\ldots,z_q)\in \R{q+1}$, we make use the following notation:
$$
A_z:=\sum_{i=0}^q z_i A_i.
$$

Now we are ready to expose the first family of examples announced in the introduction.

\begin{theorem}[Clifford cubics]\label{th1}
Let $\mathcal{A}=(A_0,\ldots, A_q)\in \mathrm{Cliff}(\R{2m},q)$. Then
\begin{equation}\label{fx}
\Phi(x)\equiv \Phi_{\mathcal{A}}(x) :=y^\TOP A_z y,   \quad x=(y,z)\in \R{2m}\oplus \R{q+1},
\end{equation}
is a radial minimal cubic satisfying (\ref{Lgamma1}) with the weight
\begin{equation}\label{radialf}
\lambda(\Phi)=-8 |x|^2.
\end{equation}
\end{theorem}

\begin{proof}
By Remark~\ref{r1}, all $A_i$ are trace free. Hence $\Delta \Phi=0$ and
\begin{equation}\label{gradf}
\Phi_{y_i}=2e_i^\TOP A_zy, \quad \Phi_{z_i}=y^\TOP A_i y,
\end{equation}
and $\Phi_{y_iy_j}=2e_i^\TOP A_ze_j$ to (\ref{Lgamma}), where $\{e_i\}_{1\le i\le 2m}$ is the standard orthonormal basis in $\R{2m}$ (we interpret $e_i$ as a vector-column). We have
\begin{equation}\label{Lf}
\begin{split}
-L(\Phi)&=8\sum_{i,j=1}^{2m} e_i^\TOP A_zy \cdot e_i^\TOP A_ze_j\cdot e_j^\TOP A_zy +
8\sum_{i=1}^{2m}\sum_{k=0}^{q}e_i^\TOP A_zy \cdot e_i^\TOP A_k y \cdot y^\TOP A_k y\\
&=: S_1+S_2.
\end{split}
\end{equation}
We find  by virtue of (\ref{Cliffsys}) that
\begin{equation}\label{linear}
\begin{split}
A_{z'}A_{z''}+A_{z''}A_{z'}&=\sum_{i,j=0}^{q} z'_iz''_j (A_iA_j+A_jA_i)=I_{\R{2m}}\sum_{i,j=0}^q 2\delta_{ij}z'_iz''_j \\
&=\scal{z'}{z''}I_{\R{2m}},
\end{split}
\end{equation}
where $\scal{\cdot}{\cdot}$ is  the standard scalar product in $\R{q}$. In particular, $A_z^2=|z|^2I_{\R{2m}}$, hence
\begin{equation}\label{cube}
A_z^3=|z|^2A_z.
\end{equation}
On the other hand, $\sum_{i=1}^{2m}e_ie_i^{\TOP}=I_{\R{2m}}$, thus we obtain
\begin{equation*}
\begin{split}
S_1&=8\sum_{i,j=1}^{2m} y^\TOP A_z (e_i e_i^\TOP) A_z (e_j {e_j}^\TOP) A_zy=8y^\TOP A_z^3y=8|z|^2\cdot y^\TOP A_zy.
\end{split}
\end{equation*}

Now applying (\ref{linear}) to $z'=z$ and $z''= (y^\TOP A_0 y, \ldots, y^\TOP A_q y),$ we find
$$
y^{\TOP}A_{z}A_{z''} y=\frac{1}{2}y^{\TOP}(A_{z}A_{z''}+A_{z''}A_{z})y=|y|^2\scal{z}{\tau}=|y|^2\cdot y^{\TOP}A_{z}y,
$$
which yields
\begin{equation*}
\begin{split}
S_2&=8\sum_{j=0}^{q}y^\TOP A_zA_j y \cdot y^\TOP A_j y=
8y^{\TOP}A_{z}A_{\tau} y=8|y|^2\cdot y^{\TOP}A_{z}y.
\end{split}
\end{equation*}
Substituting the found relations into (\ref{Lf}) yields $L(\Phi)=-8(|z|^2+|y|^2)\Phi$, hence (\ref{radialf}) is proved.

In order to show that $\Phi$ is irreducible, we assume the contrary. Then any specialization of $\Phi$ must be reducible too. For instance, by setting $z_i=0$, $2\le i \le q$, we see that the cubic form
\begin{equation}\label{isreducible}
z_0\cdot y^{\TOP}A_0y+z_1\cdot y^{\TOP}A_1y
\end{equation}
is reducible. Since  $A_0$ and $A_1$ also form a Clifford system, one can  choose new orthogonal coordinates $(u,v)\in \R{m}\times \R{m}$ such that $A_0(u,v)=(u,-v)$ and $A_1(u,v)=(v,u)$; see \cite[Lemma~5.4.7]{Baird}. Then  setting $u_i=v_i=0$, $2\le i \le m$, in (\ref{isreducible}) we conclude that the resulted  specialization  $g:=z_0(u_1^2-v_1^2)+2z_1u_1v_1$ must be reducible too. The latter expression, however, cannot be reducible even over a bigger complex polynomial ring $\mathbb{C}[z_0,z_1,u_1,v_1]$ because the discriminant of $g$ with respect to $u_1$ is $4v_1^2(z_0^2+z_1^2)$, not a perfect square. The contradiction finishes the proof.
\end{proof}

\begin{definition}
We shall call $\Phi_{\mathcal{A}}(x)$  a \textit{Clifford minimal cubic}.
\end{definition}

It follows from the general theory of Clifford systems (see, for example, \cite{Shapiro}) that a necessary and sufficient condition for existence of a symmetric Clifford system $\mathcal{A}$ in $\R{2m}$ of cardinality $q$ is  that $q\le \rho(m)$, where $\rho$ is the Hurwitz-Radon function
\begin{equation}\label{foll}
\rho(2^s\cdot \mathrm{odd})=8a+2^b, \qquad \text{where} \;s={4a+b} , \;\; 0\leq b\le 3.
\end{equation}
Thus, an irreducible Clifford minimal cubic $\Phi_{\mathcal{A}}$  does exist in $\R{n}$  precisely if  equation $n=2m+q+1$ has a solution $(q,m)$ satisfying $1\le q\le \rho(m)$. In that case we say that $n$ is \textit{realizable}, and the pair $(q,m)$ is \textit{admissible} for the number $n$.

The above results yield a partial answer to Problem~1 posed by Hsiang in \cite{Hsiang67} on the existence of irreducible cubics in a given dimension $n\ge 4$.

\begin{corollary}\label{cor:non}
Let $n\ge 4$ and $n\not\equiv 1 \mod 16$. Then there is at least one irreducible radial minimal cubic  in $\R{n}$.
\end{corollary}

\begin{proof}
It easily follows from the definition of the Hurwitz-Radon function that $\rho(2^sk)\ge 2^s$ for $s=0,1,2,3$ and any integer $k\ge1$.

If $n$ is even then $(q,m)=(1,\frac{n-2}{2})$ is admissible for $n$ because $\rho(m)\ge1$ for any $m\ge1$. Thus any even $n\ge 4$ is realizable.

Now assume that $n$ is odd. A simple verification shows that there is exactly two non-realizable values of $n$ for $n\le 16$, namely $n=5$ and $n=9$. On the other hand, in $\R{5}$ there is an isoparametric minimal cubic considered in Example~\ref{ex2} and in $\R{9}$ the Hsiang minimal cubic given in Example~\ref{ex3}. These cubics are easily shown to be radial.

Thus, we can suppose that $n$ is odd and $n\ge 16$. Write $n=16k+s$, where $k\ge 1$, $s$ is odd and $s\le 15$.
If  $s\in\{3,5,7,9\}$ then the pair $(s-1,\frac{n-s}{2})\equiv (s-1,8k)$ is admissible for $n$ because one has $\rho(8k)\ge 8$.

If $s\in\{11,15\}$ then $n$ has the form $n=4p+3$, hence the pair $(2,\frac{n-3}{2})\equiv (2,2p)$ is admissible for $n$ because $\rho(2p)\ge 2$. If $s=13$ then $n$ has the form $n=8p+5$, hence the pair $(4,\frac{n-5}{2})\equiv (4,4p)$ is admissible for $n$ because $\rho(4p)\ge 4$. The corollary is proved.
\end{proof}

\begin{remark}
A more delicate argument shows that some terms of the exceptional sequence $n=16k+1$ are really non-realizable (for example, those corresponding to $k\le 2^7=128$). On the other hand, all terms of the form $n=2^{11}\,p+17$, $p\ge 1$, are realizable.
\end{remark}

\begin{conjecture}
There is no irreducible radial minimal cubics in $\R{17}$.
\end{conjecture}

\begin{example}\label{ex4}
For any integer $m\ge 1$, the following matrices define a Clifford system on $\R{2m}$:
$$
A_0=\left(
        \begin{array}{ll}
          I_{\R{m}} & \phantom{-}0 \\
          0 & -I_{\R{m}}\\
        \end{array}
      \right),
\qquad
A_1=\left(
        \begin{array}{ll}
          0 & I_{\R{m}} \\
          I_{\R{m}} & 0\\
        \end{array}
      \right).
$$
Setting $y=(x_1,\ldots, x_{2m})$, $z=(x_{2m+1},x_{2m+2})$, we obtain by  Theorem~\ref{th1} the following irreducible radial minimal cubic in $\R{2m+2}$:
$$
\Phi_{\mathcal{A}}=x_{2m+1}(x_1^2+\ldots+x_{m}^2-x_{m+1}^2-\ldots-x_{2m}^2)+2x_{2m+2}(x_1x_{m+1}+\ldots+x_{m}x_{2m}).
$$
For $m=1$, we get the Lawson minimal cubic mentioned in Example~\ref{ex1} above.
\end{example}

\section{Congruent Clifford minimal cubics}
\label{sec:4}

Recall that two cones are called congruent if they agree under an orthogonal endomorphism in $\R{n}$.
After the new examples of minimal cubics were shown to exist, a natural question is to characterize all congruent Clifford minimal cubics in $\R{n}$. The aim of this section is to show that the congruence classes of Clifford minimal cubics are in one to one correspondence with those of the geometrically equivalent Clifford systems.

We begin with mentioning several well-known basic facts about representation theory of Clifford systems \cite[Section~5.5]{Baird}. Recall that two  Clifford systems $\mathcal{A}=(A_0,\ldots, A_q)$ and $\mathcal{B}=(B_0,\ldots, B_q)$ in  $\mathrm{Cliff}(\R{2m},q)$ are called \textit{geometrically equivalent} if there is an orthogonal endomorphism $a$ of $\R{2m}$ such that
\begin{equation}\label{geom1}
S(\mathcal{A})=a^\TOP \, S(\mathcal{B})\, a,
\end{equation}
where $S(\mathcal{A})=\{A_z: |z|=1\}$ is the unit sphere in the $\mathrm{span}(A_0,\ldots, A_q)$.

A Clifford system $\mathcal{A}$ on $\R{2m}$ is called irreducible if it is not possible to write $\R{2m}$ as a direct sum of two non-trivial subspaces that are invariant under all of the $A_i$. Then it is well known that
\begin{itemize}
\item
each Clifford system is geometrically equivalent to a direct sum of irreducible Clifford systems;
\item
an irreducible Clifford system $(A_0,\ldots, A_q)$ on $\R{2m}$ exists precisely when $m=2^s$ and
\begin{equation}\label{rhh}
\rho(\frac{m}{2})<q\le \rho(m),
\end{equation}
where $\rho$ is the Hurwitz-Radon function (\ref{foll}).
In particular, for small values $q$, the possible pairs $(q,m)$ are
\begin{table}[ht]
\renewcommand\arraystretch{1.5}
\noindent\[
\begin{array}{|c|c|c|c|c|c|c|c|c|c|c|c|}
\hline
\; q \;&1 &2 &3 &4 &5 &6 &7 &8 &9 &10 &11 \\
\hline
\;m \;& 1 & 2 & 4 &4 &8 &8 &8 &8 &16 &32 &64\\\hline
\end{array}
\]
\end{table}
\item
For any $q\ge 1$ there is exactly one class of geometrically equivalent irreducible systems.
\end{itemize}

Our first observation is the following formula for the cardinality $q$ of the Clifford system $\mathcal{A}$ in $\Phi_{\mathcal{A}}$.

\begin{proposition}\label{pro:tau}
\begin{equation}\label{tau}
\tau(\Phi_{\mathcal{A}}):=\frac{|x|^2\trace H^3(\Phi_{\mathcal{A}})}{24L(\Phi_{\mathcal{A}})}=q-1,
\end{equation}
where $H(f)$ denotes the Hessian matrix of $f$.
\end{proposition}

\begin{proof}
Write the Hessian matrix of $\Phi:=\Phi_{\mathcal{A}}$ in the block form
$$
H(\Phi)\equiv
\left(
        \begin{array}{ll}
          \Phi_{yy} & \Phi_{yz} \\
          \Phi_{zy} & \Phi_{zz}\\
        \end{array}
      \right)
=\left(
        \begin{array}{ll}
          2A_z\,\, & Q \\
          Q^\TOP& 0 \\
        \end{array}
      \right),
$$
where $Q$ is the matrix with entries $Q_{ij}= \Phi_{y_iz_j}=2e_i^{\TOP}A_j y$.
Thus
\begin{equation}\label{traceformula}
\trace H^3(\Phi)=8\trace A_z^3+6\trace A_zQQ^{\TOP}.
\end{equation}
By virtue of (\ref{cube}) and Remark~\ref{r1}, the first term in (\ref{traceformula}) is zero. In order to determine the second term, we note that
\begin{equation}\label{byv}
\begin{split}
\trace A_zQQ^{\TOP}&=\sum_{i,j=1}^{2m}(A_z)_{ij}(QQ^{\TOP})_{ij}=\sum_{i,j=1}^{2m}(A_z)_{ij}\sum_{k=0}^{q} Q_{ik}Q_{jk}\\
&=4\sum_{i,j=1}^{2m}(A_z)_{ij}\sum_{k=0}^{q}e_i^{\TOP}A_k y\cdot e_j^{\TOP}A_k y\\
&=4\sum_{k=0}^{q}y^{\TOP}A_k(\sum_{i,j=1}^{2m}(A_z)_{ij}\,e_i\cdot e_j^{\TOP})A_ky\\
&=4\sum_{k=0}^{q}y^{\TOP}A_kA_zA_ky.
\end{split}
\end{equation}
Note also that $A_k^3=A_k$, hence on applying (\ref{linear}) we obtain
\begin{equation*}
\begin{split}
\sum_{k=0}^{q}A_kA_zA_k&=\sum_{k=0}^{q}(A_kA_z+A_zA_k)A_k-\sum_{k=0}^{q}A_zA_k^2\\
&=2\sum_{k=0}^{q}z_kA_k^3-(q+1)A_z=(1-q)A_z,\end{split}
\end{equation*}
which by virtue of (\ref{byv})  and (\ref{traceformula}) yields
\begin{equation}\label{formula}
\trace H^3(\Phi)=24(1-q)\Phi.
\end{equation}
The latter relation yields the required formula (\ref{tau}) by virtue of (\ref{radialf}).
\end{proof}

\begin{corollary}\label{cor:tau}
Consider two Clifford minimal cubics in $\R{n}$
\begin{equation}\label{two}
\Phi_{\mathcal{A}}(x)=\sum_{i=0}^q z_i y^\TOP A_iy,
\quad \Phi_{\mathcal{B}}(X)=\sum_{j=0}^Q Z_i Y^\TOP B_iY,
\end{equation}
where $x=(y,z)\in \R{2m}\oplus \R{q}\approx \R{n}$, $X=(Y,Z)\in \R{2M}\oplus \R{Q}\approx \R{n}$,
and $\mathcal{A}\in \mathrm{Cliff}(\R{2m},q)$, $\mathcal{B}\in \mathrm{Cliff}(\R{2M},Q)$.

If the cubics $\Phi_{\mathcal{A}}(x)$ and $\Phi_{\mathcal{B}}(X)$ are congruent then $q=Q$.
\end{corollary}

\begin{proof}
Notice that for any homogeneous cubic polynomial $f$, $\tau(f)$ is an invariant operator under orthogonal transformations and dilatations:
\begin{equation*}\label{inva}
 \tau(f(x))=\tau(f(U(x))), \qquad
 \tau(f(cx))=\tau(f(x)).
\end{equation*}
Indeed, the first property follows from the facts that both the operator $L$ (see, for example,  \cite[Lemma~2]{Hsiang67}) and the trace $\trace H^3(f)$ are invariant under orthogonal transformations. The second property follows from the homogeneity of $f$.
Thus $\tau(\Phi_{\mathcal{A}})=\tau(\Phi_{\mathcal{B}})$, which implies by (\ref{tau}) that $q=Q$, the corollary  is proved.
\end{proof}

\begin{theorem}[Congruence  criteria]\label{th:clif}
The Clifford cubics $\Phi_{\mathcal{A}}(x)$ and $\Phi_{\mathcal{B}}(X)$ given by (\ref{two}) are congruent in $\R{n}$ if and only
the  Clifford systems $\mathcal{A}$ and $\mathcal{B}$ are geometrically equivalent.

\end{theorem}

\begin{proof}
We first prove the "if" part. Then $q=Q$, and $(A_0,\ldots, A_q)$ and $(B_0,\ldots, B_q)$ are geometrically equivalent. Denote by $a$ the corresponding orthogonal endomorphism  of $\R{2m}$  in (\ref{geom1}). Then
\begin{equation}\label{orthog1}
A_i =a^\TOP B_{w_i} a,\qquad 0\le i\le q,
\end{equation}
where $|w_i|=1$, $w_i\in \R{q+1}$. We have $B_{w_i}=\sum_{k=0}^q w_{ik} B_k$, where
$\sum_{k=1}^q w_{ik}^2=1$. By virtue of (\ref{orthog1}) and (\ref{Cliffsys}),
\begin{eqnarray}
2\delta_{ij}I_{\R{2m}}&=&A_iA_j+A_jA_i=\sum_{k,l=0}^q w_{ik}w_{jl}\cdot a^\TOP B_kB_la \nonumber\\
&=&2I_{\R{2m}}\sum_{k,l=0}^q w_{ik}w_{jl}\delta_{kl}=2I_{\R{2m}}\sum_{k=0}^q w_{ik}w_{jk},\nonumber
\end{eqnarray}
thus $d:=(w_{ij})_{0\le i,j\le q}$ is an orthogonal matrix. Furthermore, by (\ref{orthog1})
\begin{equation}\label{orthog2}
A_z \equiv \sum_{k=0}^q z_{i} A_i =a^\TOP(\sum_{k=0}^q w_{ik}z_iB_k)a\equiv a^\TOP B_{dz}a,
\end{equation}
and thus $y^\top A_z y=(a y)^\TOP\, B_{dz}\, a y$. Write $\R{2m+q+1}\cong V_y\oplus V_z$ according to the vector decomposition $x=y\oplus z$. Then  $\Phi_{\mathcal{A}}(x)=\Phi_{\mathcal{B}}(Ux)$, where $U$ is  an orthogonal endomorphism of $V_y\oplus V_z$:
$$
U=\left(
        \begin{array}{cc}
          a & 0 \\
          0 & d \\
        \end{array}
      \right),
$$
which yields that $\Phi_{\mathcal{A}}$ and $\Phi_{\mathcal{B}}$ are congruent.

Conversely, suppose $\Phi_{\mathcal{A}}$ and $\Phi_{\mathcal{B}}$ are congruent Clifford minimal cubics given by (\ref{two}). Then $q=Q$ by Corollary~\ref{cor:tau}, so that the congruence relation (after scaling by a constant factor, if needed) reads as
\begin{equation}\label{EQUI}
\Phi_\mathcal{A}(x)=\Phi_\mathcal{B}(Ux),
\end{equation}
where $U$ is an orthogonal endomorphism of $\R{2m+q+1}$. Let $\R{2m+q+1}=V_y\oplus V_z$ be the decomposition associated with $\mathcal{A}$ and $x=y+ z$. Then $U$ can be written in the block form as follows
$$
U=\left(
        \begin{array}{ll}
          a & b \\
          c & d\\
        \end{array}
      \right), \qquad UU^\TOP=U^\TOP U=I_{\R{n}}.
$$
It follows from (\ref{EQUI}) that
\begin{equation}\label{findfrom}
|\nabla \Phi_\mathcal{A}|^2|_{(y,z)} = |\nabla \Phi_\mathcal{B}|^2|_{U(y,z)}.
\end{equation}
By using (\ref{gradf}),
$$
|\nabla \Phi_\mathcal{A}|^2 = 4|A_z y|^2+\sum_{i=0}^q (y^\TOP A_i y)^2=4|y|^2|z|^2+\sum_{i=0}^q (y^\TOP A_i y)^2,
$$
and on applying (\ref{findfrom}), the relation (\ref{findfrom}) reads as follows:
\begin{equation}\label{nab}
4|y|^2|z|^2+\sum_{i=0}^q (y^\TOP A_i y)^2=
4|ay+bz|^2|cy+dz|^2+\sum_{i=0}^q ((ay+bz)^\TOP B_i (ay+bz))^2.
\end{equation}
Setting $z=0$ in the latter identity yields $4|bz|^2|dz|^2+\sum_{i=0}^q ((bz)^\TOP B_i (bz))^2=0$, which is equivalent to the system
\begin{eqnarray}\label{nab1}
&&4|bz|^2|dz|^2=0,\\
\label{nab2}
&&(bz)^\TOP B_i (bz)=0, \quad 0\le i\le q.
\end{eqnarray}
Since the polynomial ring over $\R{}$ contains no zero divisors, (\ref{nab1}) implies that either (i) $bz\equiv 0$ or (ii) $dz\equiv 0$.

Consider first (ii). Then $d\equiv 0$, and (\ref{nab2}) additionally yields for any $i$, $0\le i\le q$, that $v^\TOP B_i v\equiv 0$ for any $v\in W:=b(V_z)$. Since $B_i$ is a symmetric endomorphism, we conclude that the restriction $B_i|_{W}\equiv 0$; in particular, $B_ibz\equiv 0$ for any $z\in V_z$. Since $d=0$, we find from the orthogonality relation $U^\TOP U=1_{\R{n}}$  that $a^\TOP b\equiv 0$. Thus (\ref{nab}) becomes
\begin{equation}\label{nab3}
4|y|^2|z|^2+\sum_{i=0}^q (y^\TOP A_i y)^2=
4(|ay|^2+|bz|^2)|cy|^2+\sum_{i=0}^q ((ay)^\TOP B_i (ay))^2.
\end{equation}
By homogeneity, $|y|^2|z|^2=|bz|^2|cy|^2$, thus $b^\TOP b=I_{V_z}$ and
$c^\TOP c=I_{V_y}$. On the other hand, by the orthogonality relation $U^\TOP U=1_{\R{n}}$ we have  $a^\TOP a+c^\TOP c=I_{V_y}$, so that $a^\TOP a=0$. It follows that   $a\equiv 0$ and substituting this into (\ref{nab3}) yields
$$
\sum_{i=0}^q (y^\TOP B_i y)^2=0,
$$
a contradiction.

Consider now the remaining alternative (i). Then $b\equiv 0$ and from the orthogonality relations $UU^\TOP=U^\TOP  U =1_{\R{n}}$ we infer $aa^\TOP=I_{V_y}$, $c=0$ and  $d^\TOP d=I_{V_z}$. Therefore $a$ (resp. $d$) is an orthogonal endomorphism of $V_y$ (resp. of $V_z$). Substituting this into (\ref{EQUI}) yields
\begin{equation}\label{nab4}
y^\TOP A_z y=(ay)^\TOP B_{dz} (ay).
\end{equation}
Since both $A_z$ and $B_{dz}$ are symmetric matrices, the latter identity implies $A_z=a^\TOP B_{dz} a$. Since $a$ and $d$ are orthogonal transformations, we arrive to (\ref{geom1}), thus $\mathcal{A}$ and $\mathcal{B}$ are geometrically equivalent. The theorem is proved completely.
\end{proof}

Combining this with the facts mentioned in the beginning of this section, one obtains the following values for the number of distinct congruence classes  of irreducible  minimal Clifford cubics in $\R{n}$ for $n\le 21$ given in  Table~\ref{tab2} below.
\begin{small}
\begin{table}[h]
\renewcommand\arraystretch{1.5}
\noindent\[
\begin{array}{c|c|c|c|c|c|c|c|c|c|c|c|c|c|c|c|c|c|c}
\; n \;&4 &5 &6 &7 &8 &9 &10 &11 &12 &13 &14 &15 &16 &17 &18 &19 &20 &21\\
\hline
\;c(n)  \;& 1 & 0 & 1 &1 &1 &0 &1 &1 &2 &1 &1 &1 &1 &0 &1 &1 &2 &2\\
\end{array}
\]
\caption{}\label{tab2}
\end{table}
\end{small}

\section{Minimal cones of arbitrary high degree}\label{sec:det}

In \cite[Problem~2]{Hsiang67}, Hsiang asks  the following question. For a given dimension $n\ge 4$, are there irreducible homogeneous polynomials in $n$ real variables of \textit{arbitrary high degree}, which give minimal cones of codimension one in $\R{n}$? Below we give a partial answer on this question  by showing that there are irreducible  minimal cones of arbitrary high degree.

Let $X=(x_{ij})_{1\le i,j\le m}$ be an $m$-by-$m$ square matrix, $m\ge 1$. Consider the determinantal function
$$
\Psi_m(X)=\det (x_{ij})
$$
which is an element of the polynomial ring $\R{}[x_{11},\ldots,x_{mm}]$.
It is well known that $\Psi_m(X)$ is a irreducible polynomial in the polynomial ring $\mathbb{C}[x_{11},\ldots,x_{mm}]$ (see, for example, \cite[\S~ 30]{Waerden}, \cite[p.~630]{Wallach}).

\begin{theorem}\label{thm:det}
$\Psi_m(X)$ is an irreducible eigenfunction of $L$ of degree $m$.
\end{theorem}

\begin{proof}
The case $m=1$ is trivial and for $m=2$ the theorem is equivalent to the well-known fact that the Clifford cone $x_{11}x_{22}-x_{12}x_{21}=0$ is a minimal submanifold in $\R{4}$. Thus, we may suppose that $m\ge 3$.

We have $\partial_{x_{ij}}\Psi_m=(-1)^{i+j}\det X^{i|j}.$
Here and in what follows, $X^{\alpha|\beta}$ denotes the submatrix of $X$ obtained by deleting the rows indexed by $i\in \alpha $ and columns indexed by $j\in \beta$ for the index sets $\alpha,\beta\subset\{1,\ldots,m\}$. Similarly, $X_{\alpha|\beta}$ denotes the submatrix of $X$ that lies in the corresponding rows and columns. Note that in the given notation, $x_{ij}=X_{i|j}$.

Thus,  the second derivative is found as:
\begin{equation}\label{last1}
\partial^2_{x_{ij}x_{kl}} \Psi_m=(-1)^{i+j+k+l}
        \epsilon_{i,j;k,l}\cdot \det X^{i,k|j,l},
\end{equation}
where $\epsilon_{i,j;k,l}$ is the sign of $(i-k)(j-l)$ when the product is nonzero and $\epsilon_{i,j;k,l}=0$ otherwise.
Since $\Psi_m$ is linear in each variable, $\Delta \Psi_m=0$, thus
\begin{equation}\label{last}
\begin{split}
L(\Psi_m)&=-\sum\nolimits^* \partial^2_{x_{ij}\,x_{kl}} \Psi_m \, \partial_{x_{ij}} \Psi_m\partial_{x_{kl}} \Psi_m\\
&=-\sum\nolimits^{*}\epsilon_{i,j;k,l}\det X^{i,k|j,l}
\det X^{i|j}\det X^{k|l}\\
\end{split}
\end{equation}
where $\sum\nolimits^{*}$ denotes the sum over all indices satisfying $(i-k)(j-l)\ne 0$.

On the other hand, $\epsilon_{i,j;k,l}=-\epsilon_{i,l;k,j}$ and $\det X^{i,k|j,l}=\det X^{i,k|l,j}$, hence
\begin{equation}\label{last0}
\begin{split}
L(\Psi_m)=-&\sum\nolimits^* \partial^2_{x_{il}\,x_{kj}} \Psi_m \, \partial_{x_{il}} \Psi_m\partial_{x_{kj}} \Psi_m\\
=-&\sum\nolimits^{*}\epsilon_{i,l;k,j}\det X^{i,k|l,j}
\det X^{i|l}\det X^{k|j}\\
=\phantom{-}&\sum\nolimits^{*}\epsilon_{i,j;k,l}\det X^{i,k|j,l},
\det X^{i|l}\det X^{k|j}\\
\end{split}
\end{equation}
thus
\begin{equation}\label{last2}
\begin{split}
L(\Psi_m)&=-\frac{1}{2}\sum\nolimits^{*}{\epsilon_{i,j;k,l}}
(\det X^{i|j}\det X^{k|l}-\det X^{i|l}\det X^{k|j})\det X^{i,k|j,l}\\
&=
-\frac{1}{2}\sum\nolimits^{*}{\epsilon_{i,j;k,l}}
\left|
        \begin{array}{cc}
        \det X^{i|j} &  \det X^{i|l}\\
        \det X^{k|j} & \det X^{k|l}
        \end{array}
      \right|\cdot\det X^{i,k|j,l}.
\end{split}
\end{equation}
By using the representation of the inverse matrix $X^{-1}$ in terms of its adjoint,
$$
(X^{-1})_{ij}=\frac{(-1)^{i+j}}{\det X}\,\det X^{i|j},
$$
we find for the  $2\times 2$ determinant in (\ref{last2}),
\begin{equation*}\label{last3}
\begin{split}
\left|
        \begin{array}{cc}
        \det X^{i|j} &  \det X^{i|l}\\
        \det X^{k|j} & \det X^{k|l}
        \end{array}
      \right|&=(-1)^{i+j+k+l}(\det X)^2\cdot\left|
        \begin{array}{cc}
         \det(X^{-1})_{ij} &  \det(X^{-1})_{kj}\\
          \det(X^{-1})_{il} & \det(X^{-1})_{kl} \\
        \end{array}
      \right|\\
&\equiv \vphantom{\left|
        \begin{array}{cc}
         \det(X^{-1})_{ij} &  \det(X^{-1})_{kj}\\
          \det(X^{-1})_{il} & \det(X^{-1})_{kl} \\
        \end{array}
      \right|}
(-1)^{i+j+k+l}(\det X)^2\cdot \epsilon_{i,j;k,l}\det (X^{-1})_{ik|jl}
\end{split}
\end{equation*}

On the other hand, there is a formula relating the minors of $X^{-1}$ to those of $X$ (see \cite[\S~0.8.4]{Horn}), which reads in our notation as follows:
$$
\det (X^{-1})_{\alpha|\beta}\det X=(-1)^{|\alpha|+|\beta|}\det X^{\beta|\alpha},
$$
where $|\alpha|=\sum_{i\in \alpha} i$. Combining this with (\ref{last2}) we obtain
\begin{equation}\label{last4}
\begin{split}
L(\Psi_m)&=-\frac{\det X}{2}\sum\nolimits^{*} (\det X^{ik|jl})^2.
\end{split}
\end{equation}
Thus, we see that $\Psi_m$ satisfies (\ref{Lgamma1}) with $\lambda(\Psi_m)=-\frac{1}{2}\sum\nolimits^{*} (\det X^{ik|jl})^2$.

\end{proof}

\begin{example}\label{ex:radial9}
Consider the case $m=3$ in Theorem~\ref{thm:det}. Then
$$
\Psi_3=\det
\left(
  \begin{array}{ccc}
    x_{1} & x_2 &x_3 \\
    x_4 & x_5 & x_6 \\
    x_7 & x_8 & x_9 \\
  \end{array}
\right)=x_1x_5x_9+x_2x_6x_7+x_3x_4x_8-
x_3x_5x_7+x_2x_4x_9+x_1x_6x_8,
$$
and the corresponding weight function is found as:
$$
\lambda(\Psi_3)=-\frac{1}{2}\sum\nolimits^{*} (\det X^{ik|jl})^2=-\frac{1}{2}\sum_{i=1}^9 x_i^2.
$$
The latter relation shows that $\Psi_3$ is a radial minimal cubic in $\R{9}$.  In fact, it can be shown by a direct computation that the Hsiang cubic $b_3(X)$ from Example~\ref{ex3} is congruent to $\Psi_3$.

\end{example}

\section{Concluding remarks}

There is some formal resemblance  between an appearance of Clifford systems in the above constructions (cf. (\ref{fx})) and in the well-known examples of the isoparametric hypersurfaces with four principal curvatures given by Ferus, Karcher and Mz\"unzner in \cite{FKM}. Recall that the latter are the level-sets in the unit sphere of following quartic polynomials:
$$
F(y) =|y|^4-2\sum_{i=0}^q (y^\TOP A_i y)^2, \qquad y\in \R{2m}.
$$
Since these isoparametric hypersurfaces generates minimal quartics (as focal varieties), it would be interesting
to learn whether there is any reasonable connection between the FKM-quartics and the Clifford minimal cubics constructed on the present paper (cf. \cite{Sol}).

Another curious observation is that all the irreducible minimal cubics mentioned in Examples~\ref{ex1}--\ref{ex3}, as well as the new examples of Clifford minimal cubics constructed in Section~\ref{sec:3},  are the \textit{radial} minimal cubics. On the other hand, \textit{reducible}  minimal cubics need not to be radial as the following example shows. Consider a reducible cubic $f=x_{6}(2x_1^2+2x_2^2-x_3^2-x_4^2-x_5^2)$. Then it is an eigenfunction with the weight
$$
\lambda(f)=-28(x_1^2+x_2^2)-10(x_3^2+x_4^2+x_5^2)-16x_6^2.
$$
In a forthcoming paper \cite{Tk2010} we study the general radial cubics in more detail. Some further observations make the following conjecture is plausible.

\begin{conjecture}
Any minimal \textit{irreducible} cubic is radial.
\end{conjecture}

Note also that among all minimal cubics considered here, only isoparametric minimal cubics given in Example~\ref{ex2} are properly immersed, i.e. the corresponding defining polynomial $f$ satisfies the non-degenerating property $|\nabla f|\ne 0$ on $\mathcal{F}(f)\setminus \{0\}$. It  follows from a recent result of O.~Perdomo \cite{Perdomo1} that there are no properly immersed minimal cubics in $\R{4}$. It would be interesting to
know if this property extends to higher dimensions.

\subsection*{Acknowledgment}
I would like to thank the referee for offering useful comments and suggestions.

\bibliographystyle{amsunsrt}

\end{document}